\documentclass[11pt,a4paper]{amsart}
\usepackage{a4wide,amsfonts,amsmath,latexsym,amssymb,euscript,eufrak,graphicx,units,mathrsfs,bbold}   

\usepackage{amsmath, amsthm, amsfonts, amssymb,color}

\def\HH{\EuFrak H}
\def\RR{\mathbb R}
\def\R{\mathbb R}
\def\E{\mathbb E}
\def\EE{\mathbb E}
\def\DD{\mathbb D}
\def\Z{\mathbb Z}
 
\def\Var{{\rm Var}}
\def\dom{{\rm Dom}}

\newtheorem{prop}{Proposition}[section]
\newtheorem{proposition}{Proposition}[section]

\newtheorem{lemma}[prop]{Lemma}

\newtheorem{theorem}[prop]{Theorem}

\newtheorem{remark}[prop]{Remark}

\numberwithin{equation}{section}
\title[Improved rates in the Breuer-Major theorem]{ The Breuer-Major Theorem in total variation: \\ improved rates under minimal regularity}

\author{Ivan Nourdin}
\address{Ivan Nourdin, Universit\'e du Luxembourg, 
Unit\'e de Recherche en Math\'ematiques,
Maison du Nombre,
6 avenue de la Fonte,
L-4364 Esch-sur-Alzette,
Grand Duch\'e du Luxembourg}
\email{ivan.nourdin@uni.lu} 
\thanks{I. Nourdin was supported by the FNR grant APOGee at Luxembourg University, D. Nualart by the NSF grant  DMS 1811181, and G. Peccati
 by the FNR grant FoRGES (R-AGR-3376-10) at Luxembourg University.}
\author{David Nualart} 
\address{David Nualart,
Department of Mathematics, University of Kansas, 405 Snow Hall, Lawrence, Kansas, 66045, USA}
\email{nualart@ku.edu}
\author{Giovanni Peccati}
\address{Giovanni Peccati, Universit\'e du Luxembourg, 
Unit\'e de Recherche en Math\'ematiques,
Maison du Nombre,
6 avenue de la Fonte,
L-4364 Esch-sur-Alzette,
Grand Duch\'e du Luxembourg}
\email{giovanni.peccati@uni.lu}

\date{\today}

\begin{document}
  
  \maketitle 
   
   \begin{abstract}{In this paper we prove an estimate for the total variation distance, in the framework of the
   Breuer-Major theorem, using the Malliavin-Stein method,  assuming the underlying function $g$ to be once weakly differentiable with $g$ and $g'$ \textcolor{black}{having} finite moments of order four with respect to the standard Gaussian density. This result is proved by a combination of Gebelein's inequality and some novel estimates involving Malliavin operators.

\noindent{{\bf Keywords: } Breuer-Major theorem; Integration by Parts; Rate of Convergence; Malliavin-Stein approach.
}}
\end{abstract}
   
 \section{Introduction}
 
 \subsection{Overview and main findings}
 Let $X=\{ X_n , n\ge 0\}$  be a real-valued centered stationary Gaussian sequence with unit variance, that we assume to be defined on an appropriate probability space $(\Omega, \mathscr{F}, \mathbb{P})$. For $k \in \Z$, set $\rho(k) := \E(X_0 X_k)$ if $k\ge 0$, and $\rho(k) := \rho(-k)$ if $k<0$. 
  Denoting by $\gamma(dx) = (2\pi)^{-1/2} e^{-x^2/2}dx$ the standard Gaussian measure on the real line,  we say that a function $g \in L^2( \mathbb{R}, \gamma)  =: L^2(\gamma)$ has {\bf Hermite rank} $d\ge 1$ if
	\begin{equation}\label{hexp}
g(x)= \sum_{q=d} ^\infty c_q  H_q(x),
\end{equation}
 where $c_d \not =0$, $H_q$ is the $q$th Hermite polynomial (to be formally defined in Section \ref{ss:mall}), and the series converges in $L^2(\gamma)$. The forthcoming Theorem \ref{t:cbm} --- known as the {\bf Breuer-Major Theorem}  (see \cite{bm}, as well as \cite{Ta}) --- establishes a sufficient condition for the sequence  
 \begin{equation}\label{fn}
 F_n := \frac{1}{\sqrt{n}} \sum_{i=1}^n g(X_i) , \quad n\geq 1,
\end{equation} 
to verify a Central Limit Theorem (CLT). 

\smallskip

\noindent{\bf Remark on notation.} From now on, we write $N(\mu, \tau^2$) to indicate a generic random variable with mean $\mu$ and variance $\tau^2$. We also put $N_\tau:=N(0,\tau^2)$ and
for $\tau =1$, $N=N_1$ denotes a standard normal Gaussian variable.  The symbol $\Rightarrow$ denotes convergence in distribution of random elements. Given two real-valued random variables $X,Z$, the {\bf total variation distance} between the distributions of $X$ and $Z$ is defined as
\begin{equation}\label{e:dtv}
d_{\rm TV}(X,Z) := \sup_{A} \left| \,  \mathbb{P} (X\in A)- \mathbb{P} (Z\in A) \, \right|,
\end{equation}
where the supremum runs over the class of all Borel subsets $A$ of $\R$. Depending on notational convenience, given a numerical sequence $\{\alpha(k) : k\in \Z \}$, we will often write $\sum \alpha(k)$ to indicate the full sum $\sum_{k\in \Z} \alpha(k)$, whenever it is well-defined. Finally, given a random variable $Z$, we use the notation $\|Z\|_q = \E[|Z|^q]^{1/q}$, for every $q>1$. 

\smallskip

 \begin{theorem}[\bf Breuer-Major]\label{t:cbm} Let $g\in L^2(\gamma)$ have Hermite rank $d\geq 1$, and assume moreover that \begin{equation} \label{h1}
\sum_{j \in \Z} |\rho(j)|^d < \infty.
\end{equation}
Then, as $n\to\infty$, 
\begin{equation}\label{e:cbmc}
F_n \Rightarrow N(0, \sigma^2),
\end{equation}
where
\begin{equation}\label{bm.sig}
\sigma^2 := \sum_{q=d}^\infty q! c_q^2 \sum_{k \in \Z} \rho(k)^q < \infty.
\end{equation}
\end{theorem}
 
 Theorem \ref{t:cbm} is one of the staples of modern Gaussian analysis, with far-reaching applications ranging from stochastic geometry to mathematical statistics and information theory --- see e.g. \cite{Doukhan, np-book, PiTa, Tudor} for a general discussion, as well as \cite{BBNP, CNN, CS, NourdinNualartPTRF, np-book, NP, NPP, NourdinPeccatiRossi} for a sample of recent extensions and ramifications. 
 
 \smallskip
 
Using the fact that the limiting random variable $N(0,\sigma^2)$ has a density, it is straightforward to deduce from the second Dini's theorem that the convergence \eqref{e:cbmc} always takes place in the sense of the {\bf Kolmogorov distance}, that is:  with the notation $N_\sigma = N(0,\sigma^2)$,
$$
\sup_{t\in \R} \left | \, \mathbb{P}[F_n \leq  t] - \mathbb{P}[N_\sigma \leq t] \, \right| \longrightarrow 0, \quad n\to\infty. 
$$
On the other hand, determining wether \eqref{e:cbmc} takes place in the sense of the total variation distance \eqref{e:dtv} is a more delicate matter, for which no exhaustive criterion is currently known. The difficulty of such an issue is demonstrated by considering the following two facts, corresponding to choices of the function $g$ in the Breuer-Major Theorem yielding contrasting behaviours with respect to $d_{\rm TV}$:

\begin{itemize}

\item[\bf (a)] according to the main results of \cite{NouPol}, if $g$ in Theorem \ref{t:cbm} is a polynomial, then necessarily $d_{\rm TV}(F_n, N_\sigma) \to 0$, as $n\to \infty$;

\item[\bf (b)] if $g$  takes values in a discrete set, then (trivially) $d_{\rm TV}(F_n, N_\sigma) =1$ for every $n$.

\end{itemize}

\smallskip

The aim of the present paper is to deduce new explicit bounds on the total variation distance 
\begin{equation} \label{yn}
 Y_n:=\frac {F_n } {\sqrt{\Var (F_n)} } ,
\end{equation}
and a standard normal random variable $ N = N(0,1)$, in the case where $g$ has Hermite rank $d=2$. We will see that our estimates imply minimal regularity conditions on $g$, in order for the limiting relation $d_{\rm TV}(Y_n, N )\rightarrow 0$ (or, equivalently, $d_{\rm TV}(F_n, N(0,\sigma^2))\to 0$) to take place. Moreover, under comparable regularity assumptions on $g$, the rates of convergence provided by our bounds are better than or commmensurate to the \textcolor{black}{best estimates to date,} obtained in \cite{KN, NPY, NZ}. The main tool exploited in our analysis is a non-trivial combination of {\bf Gebelein's inequality} (recalled in Section \ref{ss:geb} below, and already used in \cite{NPY}), and some novel estimates involving Malliavin operators --- see e.g. the forthcoming Lemma \ref{lem1}. 

\smallskip

Our main findings are contained in the following statement, in which we use the notation $\DD^{k,p} (\RR, \gamma)$, $p\geq 1$, $k=1,2 \dots$,   to denote the Sobolev space given by the closure of the class of polynomials mappings $ q : \RR\to\RR$ with respect to the norm
$$
\| q \|_{k,p} = \left | \int_\RR \left( | q(x) |^p +\sum_{i=1}^k | D^i q(x)|^p  \right)\,\gamma(dx)\right|^{1/p},
$$
where $D^i$ denotes the $i$th derivative of $q$ as a function of $x$.	 

The following is the main result of this paper.
 
 \begin{theorem} \label{thm1}
Assume that $g\in L^2(\RR,\gamma)$ has Hermite rank $d= 2$ and belongs to $ \DD^{1,4}(\RR,\gamma)$. Suppose that   \eqref{h1} holds and that $\sigma^2$ defined by \eqref{bm.sig} is strictly positive.
Let $Y_n$ be the random variable defined in \eqref{yn}.  Then, there exists a constant $C>0$ independent of $n$ such that
	\begin{align} \label{ecu1}
				   d_{\rm TV}(Y_n , N)  \leq  C n^{-\frac{1}{2}} \left(\sum_{|k| \leq n} |\rho(k)|\right)^{\frac{1}{2}}  + C n^{-\frac{1}{2}} \left(\sum_{|k| \leq n} |\rho(k)|^{\frac{4}{3}}\right)^{\frac 32 }, \quad n\geq 1.
		\end{align}
\end{theorem}
 
Note that the right-hand side of \eqref{ecu1} (as well as those of the forthcoming bounds \eqref{e:npy} and \eqref{e:opt})  converges to zero, as $n\to \infty$, by virtue of Lemma \ref{l:simple}. 

\subsection{Comparison with existing results}

We will now compare Theorem \ref{thm1} with three relevant papers in the recent literature. Such a comparison exploits the log-convexity of $\ell^p$ norms, see e.g. \cite[Lemma 1.11.5]{T}:
 \begin{equation}\label{e:comp}
 \left(\sum_{|k| \leq n} |\rho (k)|^{\frac43}\right)^{\frac34}\leq  \left(\sum_{|k| \leq n} |\rho (k)|\right)^{\frac12} \left(\sum_{|k| \leq n} \rho (k)^{2}\right)^{\frac14}.
 \end{equation}

 \bigskip

 \begin{itemize}
 
 \item[\bf (1)] In \cite{NPY}, the following two facts are proved: {\bf (1a)} if $g\in \mathbb{D}^{1,4}(\mathbb{R}, \gamma)$ and has Hermite rank equal to 1, then there exists an absolute constant $C$ such that $d_{\rm TV}(Y_n,N)\leq Cn^{-1/2} $, and {\bf (1b)}
if $g\in \mathbb{D}^{1,4}(\mathbb{R}, \gamma)$ and $g$ is even, then 
\begin{equation}\label{e:npy}
d_{\rm TV}(Y_n,N)\leq Cn^{-1/2}\sum_{|k| \leq n} |\rho(k)|.
\end{equation}
In view of the usual CLT, the estimate at Point {\bf (1a)} cannot be improved. On the other hand, since an even function $g \in L^2(\gamma)$ has Hermite rank equal to 2, the estimate at \eqref{e:npy} can be meaningfully compared with our Theorem \ref{thm1}. A direct use of \eqref{e:comp} shows that, if $\rho \in \ell^1$ (that is, $\rho$ is absolutely summable), then the right-hand sides of \eqref{ecu1} and \eqref{e:npy} are both bounded by a multiple of $n^{-1/2}$, while \eqref{ecu1} is systematically smaller than \eqref{e:npy} when $\rho \notin \ell^1$. 

\smallskip

 \item[\bf (2)] Given $g = \sum c_q H_q\in L^2(\gamma)$, we define $A(g) := \sum | c_q | H_q$, that is, $A(g)$ is the element of $L^2(\gamma)$ obtained by taking the absolute value of the coefficients appearing in the Hermite expansion of $g$. In \cite{KN}, the following results are proved: {\bf (2a)} the bound 
 $$
 d_{\rm TV}(Y_n,N)\leq Cn^{-1/2}\left(\sum_{|k| \leq n} |\rho(k)|\right)^{1/2}+Cn^{-1/2}\left(\sum_{|k| \leq n}|\rho(k)|^{4/3}\right)^{3/2},
 $$ 
 holds whenever $A(g) \in \mathbb{D}^{1,4}(\R, \gamma)$ and $g$ has Hermite rank 2, and {\bf (2b)} one has the estimate
\begin{equation}\label{e:opt}
 d_{\rm TV}(Y_n,N)\leq Cn^{-1/2}\left(\sum_{|k| \leq n}|\rho(k)|\right)^{1/2}+Cn^{-1/2}\left(\sum_{|k| \leq n}|\rho(k)|^{3/2}\right)^{2},
 \end{equation}
if $A(g) \in \mathbb{D}^{2,6}(\R, \gamma)$ and $g$ has Hermite rank 2. The estimate at Point {\bf (2a)} is the same as the one appearing in our bound \eqref{ecu1}, but is obtained under the strictly stronger assumption that $A(g) \in \mathbb{D}^{1,4}(\R, \gamma)$. On the other hand, one can use the results of \cite{NP} to show that a multiple of the sequence $n\mapsto n^{-1/2}\left(\sum_{|k| \leq n}|\rho(k)|^{3/2}\right)^{2}$ also constitutes a lower bound for $n\mapsto   d_{\rm TV}(Y_n,Z)$ in the case $g = H_2$.

 \item[\bf (3)] In \cite{NZ}, the following is proved: {\bf (3a)} if $g\in \mathbb{D}^{2,4}(\R, \gamma)$ and $g$ has Hermite rank 1, then $ d_{\rm TV}(Y_n,N)\leq Cn^{-1/2}$, 
{\bf (3b)} if $g\in \mathbb{D}^{4,4}(\R, \gamma)$, and $g$ has Hermite rank $2$,  then the bound \eqref{ecu1} holds true,
 and {\bf (3c)} if $g\in \mathbb{D}^{6,8}(\R, \gamma)$ and $g$ has Hermite rank 2, then  
  \begin{equation}\label{e:ottimovero}
 d_{\rm TV}(Y_n,N)\leq Cn^{-1/2}\left(\sum_{|k|\leq n}| \rho(k)|^{3/2}\right)^2
 \end{equation}
As observed at Point {\bf (2)}, the upper bound \eqref{e:ottimovero} cannot be improved. 
 \end{itemize}
 
We would like to emphasize that, unlike  in previous works,   the bound \eqref{ecu1} for functions of Hermite rank $2$ is obtained here assuming only that $g$ is once weakly differentiable. In particular this bound holds for
 $g(x)= |x|^p - \E[|N|^p]$ for any $p\ge 1$.

 \subsection{Plan} 
 The paper is organized as follows. Section 2 contains some preliminaries on the Malliavin calculus associated with a Gaussian family of random variables and on the Malliavin-Stein method for estimating the total variation distance. We also include in this section two basic inequalities that  play an important role in the proofs: a version of  the  Brascamp-Lieb inequality and Gebelein's inequality. Section 3 is devoted to the proof of Theorem
 \ref{thm1}.

     \section{Preliminaries}
  
   In this section, we briefly recall some  elements of  the Malliavin calculus of variations associated with a Gaussian family of random variables. We refer the reader to \cite{np-book,nualartbook,CBMS} for a detailed account of this topic.  
We will also recall a crucial estimate for the total variation distance proved using the Malliavin-Stein approach, and prove two inequalities which will be used in the proof of Theorem \ref{thm1}.
   
\subsection{Malliavin calculus}\label{ss:mall}

 Let $\mathfrak{H}$ be a real separable Hilbert space; in order to simplify our discussion, we will assume for the rest of the paper that $\mathfrak{H} = L^2(A, \mathscr{A}, \mu)$, where $(A, \mathscr{A}, \mu)$ is a $\sigma$-finite measure space such that $\mu$ has no atoms. For any integer $m \geq 1$, we use the symbols $\mathfrak{H}^{\otimes m}$ and $\mathfrak{H}^{\odot m}$ to denote the $m$-th tensor product and the $m$-th symmetric tensor product of $\mathfrak{H}$, respectively. We now let $W = \{W(\phi) : \phi \in \mathfrak{H}\}$ denote an {\bf isonormal Gaussian process} over the Hilbert space $\mathfrak{H}$. This means that $W$ is a centered Gaussian family of random variables defined on $(\Omega, \mathscr{F}, \mathbb{P})$, with covariance $$\EE\left(W(\phi)W(\psi)\right) = \langle \phi, \psi \rangle_{\mathfrak{H}}, \qquad \phi, \psi \in \mathfrak{H}.$$ Without loss of generality, we can assume that $\mathscr{F}$ is generated by $W$.

\smallskip

We denote by $\mathcal{H}_m$ the closed linear subspace of $L^2(\Omega)$ generated by the random variables $\{H_m(W(\varphi)): \varphi \in \mathfrak{H}, \|\varphi\|_{\mathfrak{H}}=1\}$, where $H_m$ is the $m$-th Hermite polynomial defined by 
\[
H_m(x)=(-1)^me^{\frac{x^2}{2}}\frac{d^m}{dx^m}e^{-\frac{x^2}{2}},\quad m \geq 1,
\]
and $H_0(x)=1$. The space $\mathcal{H}_m$ is the {\bf Wiener chaos} of order $m$ associated with $W$. The $m$-th multiple integral of $\phi^{\otimes m} \in \mathfrak{H}^{\odot m}$ is defined by the identity 
$ I_m(\phi^{\otimes m}) = H_m(W(\phi))$ for any $\phi\in \mathfrak{H}$ with $\| \phi\|_{\HH}=1$. The map $I_m$ provides a linear isometry between $\mathfrak{H}^{\odot m}$ (equipped with the norm $\sqrt{m!}\|\cdot\|_{\mathfrak{H}^{\otimes m}}$) and $\mathcal{H}_m$ (equipped with $L^2(\Omega)$ norm). By convention, $\mathcal{H}_0 = \mathbb{R}$ and $I_0(x)=x$.

\smallskip

The space $L^2(\Omega)$ can be decomposed into the infinite orthogonal sum of the spaces $\mathcal{H}_m$. Namely, for any square integrable random variable $F \in L^2(\Omega)$, we have the following expansion,
\begin{equation} \label{chaos}
  F = \sum_{m=0}^{\infty} I_m (f_m) ,
\end{equation}
where $f_0 = \mathbb{E}(F)$, and $f_m \in \mathfrak{H}^{\odot m}$ are uniquely determined by $F$. The representation \eqref{chaos} is known as the {\bf Wiener chaos expansion} of $F$.  

\smallskip

  For a smooth and cylindrical random variable $F= f(W(\varphi_1), \dots , W(\varphi_n))$, with $\varphi_i \in \mathfrak{H}$ and $f \in C_b^{\infty}(\mathbb{R}^n)$ (meaning that $f$ and its partial derivatives are bounded), we define its {\bf Malliavin derivative} as the $\mathfrak{H}$-valued random variable given by
\[
 DF = \sum_{i=1}^n \frac{\partial f}{\partial x_i} (W(\varphi_1), \dots, W(\varphi_n))\varphi_i.
\]
By iteration, we can also define the $k$-th derivative $D^k F$, which is an element in the space $L^2(\Omega; \mathfrak{H}^{\otimes k})$. For any real $p\ge 1$ and any integer $k\ge 1$, the Sobolev space $\mathbb{D}^{k,p}$ is defined as the closure of the space of smooth and cylindrical random variables with respect to the norm $\|\cdot\|_{k,p}$ defined by 
\[
 \|F\|^p_{k,p} = \mathbb{E}(|F|^p) + \sum_{i=1}^k \mathbb{E}(\|D^i F\|^p_{\mathfrak{H}^{\otimes i}}).
\]
Notice that if $F=I_1(\varphi)$ is an element in the first Wiener chaos  with $\|\varphi\|_{\HH} =1$, then (using the notation introduced before Theorem \ref{thm1}) $g\in \mathbb{D}^{k,p}(\R,\gamma)$ if and only if $g(F)\in \mathbb{D}^{k,p}$. 

\smallskip

We define the {\bf divergence operator} $\delta$ as the adjoint of the derivative operator $D$. Namely, an element $u \in L^2(\Omega; \mathfrak{H})$ belongs to the domain of $\delta$, denoted by $\dom\, \delta$, if there is a constant $c_u > 0$ depending on $u$ and satisfying 
\[
|\mathbb{E} (\langle DF, u \rangle_{\mathfrak{H}})| \leq c_u \|F\|_{L^2(\Omega)}
\] for any $F \in \mathbb{D}^{1,2}$.  If $u \in \dom \,\delta$, the random variable $\delta(u)$ is defined by the duality relationship 
\begin{equation} \label{dua}
\mathbb{E}(F\delta(u)) = \mathbb{E} (\langle DF, u \rangle_{\mathfrak{H}}) \, ,
\end{equation}
which is valid for all $F \in \mathbb{D}^{1,2}$.  
In a similar way, for each integer $k\ge 2$, we define the iterated divergence operator $\delta^k$ through the duality relationship 
\begin{equation} \label{dua2}
\mathbb{E}(F\delta^k(u)) = \mathbb{E}  \left(\langle D^kF, u \rangle_{\mathfrak{H}^{\otimes k}} \right) \, ,
\end{equation}
valid for any $F \in \mathbb{D}^{k,2}$, where $u\in  {\rm Dom}\, \delta^k \subset L^2(\Omega; \mathfrak{H}^{\otimes k})$.  

\smallskip

Let $\gamma$  be the standard Gaussian measure on $\R$.  The Hermite polynomials $\{H_m(x), m\ge 0\}$  form a complete orthonormal system in $L^2(\R,\gamma)$ and  any function $g\in L^2(\R,\gamma)$ admits an  orthogonal expansion  of the form (\ref{hexp}). 
 If $g$ has Hermite rank $d$, for any integer $1\le k\le d$, we define the operator $T_k$ by
\begin{equation} \label{t1a}
T_k(g)(x) = \sum_{m=d}^\infty c_m H_{m-k}(x) \,.
\end{equation}
To simplify the notation we will write $T_k(g) =g_k$.

Suppose that $F$ is a random variable in the first Wiener chaos of $W$ of the form $F= I_1(\varphi)$, where $\varphi \in \HH$ has norm one.  Then one can check that  $g_k(F)$ has the representation
   \begin{equation}\label{g.intrep}
   	 g(F) = \delta^k( g_k(F) \varphi^{\otimes k}).
   \end{equation}
   Moreover, if  $g(F) \in \DD^{j,p}(\Omega)$ for some $j\ge 0$ and $p>1$, then $g_k(F) \in \DD^{j+k, p}(\Omega) $; in particular,  for some constant $C$ only depending on $j,k,p$, one has that 
   \begin{equation} \label{ecu2}
   \|g_k(F)\|_{j+k, p}\leq  C \|g(F)\|_{j,p}.
   \end{equation}
    We refer to \cite{NZ} for the proof of these results.

 \smallskip

The family $\{ P_t : \,t\geq 0\}$ of operators is defined for random variables $F\in L^2(\Omega)$ of the form (\ref{chaos}) via the relation $P_tF=\sum_{m=0}^\infty e^{-mt}\,I_m(f_m)$, and is called the {\bf Ornstein-Uhlenbeck semigroup} associated with $W$.
The operator $L$ is defined as $LF=-\sum_{m=0}^\infty mI_m(f_m)$, and can be shown to be the infinitesimal generator of $\{ P_t : \,t\geq 0\}$. The domain of $L$ is $\DD^{2,2}(\Omega)$ and the following {\bf Meyer inequality} holds (see \cite[Theorem 1.5.1]{nualartbook}): for any $r>1$, there exists a constant $c_r$ such that, for any $F\in \DD^{2,r}(\Omega)$,
\begin{equation}\label{meyer}
\|D^2F\|_{L^r(\Omega,\HH^{\otimes 2})}\leq c_r\,\|LF\|_{L^r(\Omega)}.
\end{equation}
We also define the operator $L^{-1}$, which is the inverse of $L$, as follows: for every $F\in L^2(\Omega)$ of the form (\ref{chaos}), we set $L^{-1}F=\sum_{m=1}^\infty -\frac1m I_m(f_m)$.

\begin{remark}\label{r:LH}{\rm Fix an integer $k\geq 1$, and consider a generic element $u$ of the class $L^2(\Omega ; \mathfrak{H}^{\otimes k})$. Then, in view of the fact that  $\HH = L^2(A, \mathscr{A}, \mu)$ by our initial assumption, it is a standard fact that $u$ admits a (parametrized) chaotic expansion of the form
$$
u(t_1,...,t_k) = \sum_{m=0}^\infty I_m (f_m(\cdot, t_1,...,t_k)),
$$
where the ($\mu^{m+k}$--almost everywhere uniquely defined)  kernels $f_m$ are square-integrable and symmetric in the first $m$ variables, and
$$
\E\left[\| u \|^2_{\mathfrak{H}^{\otimes k}}\right]  = \sum_{m=0}^\infty m! \|f_m\|^2_{L^2(\mu^{m+k})} <\infty.
$$
Using such a representation one can canonically define $L^{-1} u$ as the element of $L^2(\Omega ; \mathfrak{H}^{\otimes k} )$ given by
$$
L^{-1}u(t_1,...,t_k) = - \sum_{m=1}^\infty  \frac1m I_m (f_m(\cdot, t_1,...,t_k)).
$$
In what follows, given $k\geq 2$ and $u\in L^2(\Omega ; \mathfrak{H}^{\otimes k})$, the symbol $\widetilde{u}$ stands for the symmetrization of $u$, that is
$$
\widetilde{u}(t_1,...,t_k) =\frac{1}{k!} \sum_{\sigma} u(t_{\sigma(1)},..., t_{\sigma(k)}),
$$
where the sum runs over the group of all permutations $\sigma$ of $\{1,...,k\}$. Note that, for every $r>1$,
\begin{equation}\label{e:sym}
\| \widetilde{u}\|_{L^r(\Omega; \mathfrak{H}^{\otimes k})} \leq \| {u}\|_{L^r(\Omega; \mathfrak{H}^{\otimes k})},
\end{equation}
by the triangle inequality. Also, one has trivially that $\widetilde{L^{-1} u} = L^{-1}\widetilde{u}$.

}
\end{remark}

 We will make repeated use of the following lemma, focussing on the boundedness of $L^{-1}$.
 
 \begin{lemma} \label{lem1}
Let  $p,q,r> 1$ be such that $\frac 1p +  \frac 1q =\frac  1r$.
\begin{enumerate} 
\item Suppose that $F\in L^p(\Omega)\cap \DD^{1,4}(\Omega)$ and $G \in \DD^{1,q\vee 4}(\Omega)$. Then
$GDF$ belongs to the domain of $L^{-1}$ viewed as an $\HH$-valued operator, and
 \begin{equation}\label{ineq1}
 \|   L^{-1} (  G DF ) \|_{L^r(\Omega; \HH)} \le c_{p,q} \| F\|_p  \| G \|_{1,q}.
 \end{equation}
\item Suppose that $F\in L^p(\Omega)\cap \DD^{2,4}(\Omega)$ and $G \in \DD^{2,q\vee 4}(\Omega)$.  Then 
$GD^2F$ belongs to the domain of $L^{-1}$ viewed as an $\HH^{\otimes 2}$-valued operator, and
 \begin{equation}\label{ineq2}
 \|   L^{-1} (  G D^2F ) \|_{L^r(\Omega; \HH^{\otimes 2})} \le c_{p,q} \| F\|_p  \| G \|_{2,q}.
\end{equation}
\end{enumerate}
 \end{lemma}
 
 \begin{proof}
The proof is subdivided into several steps.
\begin{enumerate}
\item[(i)] First of all we observe that, by a direct application of the multiplier theorem (see \cite[Theorem 1.4.2]{nualartbook}), the operator $L^{-1}$ is bounded from $L^r(\Omega)$ to itself. Moreover, one can suitably modify  the proof of such a result to show that, for every $k\geq 1$, $L^{-1}$ is also bounded as an operator from $L^r(\Omega ; \HH^{\otimes k})$ to itself (see Remark \ref{r:LH}).
\item[(ii)] Let $K$ be the operator defined by $KF=\sum_{m\geq 1}\frac{m+1}{m}I_m(f_m)$ for $F=\sum_{m=0}^\infty I_m(f_m)\in L^2(\Omega)$.
Again by a direct application of the multiplier theorem (see \cite[Theorem 1.4.2]{nualartbook}), the operator $K$ is bounded from $L^r(\Omega)$ to itself. 
On the other hand, one has $-DL^{-1}=\int_0^\infty DP_t dt$ (according to \cite[Prop. 2.9.3]{np-book}) as well as the existence of $c_r>0$ such that, for any
$F\in L^{r}(\Omega)$,
\begin{equation}\label{2.15}
\|DP_tF\|_{L^r(\Omega ; \HH)}\leq c_r\frac{e^{-t}}{\sqrt{1-e^{-2t}}}\|F\|_{L^{r}(\Omega)}
\end{equation}
(according\footnote{The statement of \cite[Prop. 5.1.5]{eulalia} contains the factor $t^{-1/2}$ instead of 
$\frac{e^{-t}}{\sqrt{1-e^{-2t}}}$, but an inspection of the proof given therein actually provides the estimate stated in (\ref{2.15}).} to \cite[Prop. 5.1.5]{eulalia}); these two facts plus the Minkowski inequality imply that the operator $D(-L^{-1})$ is bounded from $L^r(\Omega)$ to $L^r(\Omega; \HH)$. 
As a conclusion, using that $-L^{-1}D = KD(-L^{-1})$, we obtain that $-L^{-1}D$ is bounded from $L^r(\Omega)$ to $L^r(\Omega;\HH)$.
\item[(iii)] Since $F\in L^p(\Omega)\cap \DD^{1,4}(\Omega)$ and $G \in \DD^{1,q\vee 4}(\Omega)$, we have that $F,G\in \mathbb{D}^{1,2}$ and $GDF, FDG, D(FG)\in L^2(\Omega ; \mathfrak{H})$. We can therefore write
 \[ 
   L^{-1} (  G DF )  = L^{-1} D(FG) - L^{-1}( FDG).
   \]
   One one hand (see point (ii) above):
   \[
   \|  L^{-1} D(FG)\|_{L^r(\Omega; \HH)}   \le c \| FG\|_r\le \| F\|_p \|G\|_q.
   \]
   On the other hand (see point (i) above):
   \[
  \| L^{-1}( FDG)\| _{L^r(\Omega; \HH)}  \le c \| FDG \|_{L^r(\Omega; \HH)}  \le c \|F\|_p \| DG\|_{L^q(\Omega ; \HH)}.
  \]
  This completes the proof of (\ref{ineq1}).
\item[(iv)] We now suppose that $F\in L^p(\Omega)\cap \DD^{2,4}(\Omega)$ and $G \in \DD^{2,q\vee 4}(\Omega)$.  
We can write
 \[ 
   L^{-1} (  G D^2F )  = L^{-1}( D^2(FG)) - 2\,L^{-1}( \widetilde{D(FDG)})+L^{-1}(FD^2G),
   \]
where the involved symmetrization is defined in Remark \ref{r:LH}. Let $M$ be the operator defined by $MZ=\sum_{m\geq 1}\frac{m+2}{m}I_m(z_m)$ for $Z=\sum_{m=0}^\infty I_m(z_m)\in L^2(\Omega)$.
By a direct application of the multiplier theorem (see \cite[Theorem 1.4.2]{nualartbook}), the operator $M$ is bounded from $L^r(\Omega)$ to itself. 
Thus, using on one hand that $-L^{-1}D^2 = MD^2(-L^{-1})$ and on the other hand that $D^2(-L^{-1})$ is bounded from $L^r(\Omega;\HH^{\otimes 2})$ to itself (by (\ref{meyer})), we obtain that 
$-L^{-1}D^2$ is bounded from $L^r(\Omega)$ to $L^r(\Omega;\HH^{\otimes 2})$.
As a consequence
   \[
   \|  L^{-1} (D^2(FG))\|_{L^r(\Omega; \HH)}   \le c \| FG\|_r\le \| F\|_p \|G\|_q.
   \]
   On the other hand (see points (i) and (ii) above, as well as \eqref{e:sym}):
   \begin{eqnarray*}
  \| L^{-1}( FD^2G)\| _{L^r(\Omega; \HH^{\otimes 2})} & \le& c \| FD^2G \|_{L^r(\Omega; \HH^{\otimes 2})}  \le c \|F\|_p \| D^2G\|_{L^q(\Omega ; \HH^{\otimes 2})}\\
  \|L^{-1}( \widetilde{D(FDG)})\| _{L^r(\Omega; \HH^{\otimes 2})} &\le& c \| FDG \|_{L^r(\Omega; \HH)}  \le c \|F\|_p \| DG\|_{L^q(\Omega ; \HH)}.
 \end{eqnarray*}
  This completes the proof of (\ref{ineq2}).
\end{enumerate}
 \end{proof}

\subsection{Stein's method}
We refer to  \cite{ChenGoldShao} for a complete discussion of this topic.
Let  $h: \mathbb{R} \to \mathbb{R}$ be a Borel function such that $h \in L^1(\R, \gamma)$ and let $N\sim d\gamma(x)$. The ordinary differential equation
      \begin{equation} \label{stein}
      f'(x) - xf(x) = h(x) - \mathbb{E}(h(N))
      \end{equation}
is called the Stein's equation associated with $h$. The function 
\[
f_h(x):= e^{x^2/2}\int_{-\infty}^x (h(y) - \mathbb{E}(h(N)))e^{-y^2/2} dy
\]
 is the unique solution to the Stein's equation satisfying $\lim_{|x| \to \infty} e^{-x^2/2} f_h(x) = 0$. Moreover, if $h$ is bounded by $1$, then $f_h$ satisfies  $\|f_h \|_\infty \leq \sqrt{\pi /2}$ and 
    $   \|f_h' \|_\infty \leq 2$. 
We refer  to  \cite{np-book} and the references therein for a complete proof of these results. 

We recall  the total variation distance between the laws of two random variables  defined 
in (\ref{e:dtv}). 
  Substituting $x$ by $F$  in  Stein's equation (\ref{stein}) and
 using the   estimate for  $ \|f_h' \|_\infty$    lead  to the fundamental estimate
\begin{equation}  \label{equ83}
d_{\rm TV}(F,N) \le  \sup_{f\in  \mathcal{C}^1(\R),    \| f' \|_\infty \le 2 } | \EE  (f'(F)- Ff(F)) | \,.
\end{equation}

In the framework of an isonormal Gaussian process $W$, we  can use Stein's equation to estimate  the  total variation distance between a random variable $F = \delta(u)$ and $N$.  A basic result is given in the next proposition (see \cite{eulalia,np-book}), which is an easy consequence of  (\ref{equ83}) and the duality relationship (\ref{dua}).
 
\begin{proposition}\label{prop2}
	Assume that $u\in {\rm Dom} \,\delta$,  $F=\delta(u) \in \mathbb{D}^{1,2}$ and $\EE(F^2)=1$.  Then,
	\begin{eqnarray*}
	d_{\rm TV} (F,N) 	\le   2 \sqrt{\Var( \langle DF, u \rangle_{\mathfrak{H}} )}  \,.
	\end{eqnarray*}
\end{proposition}

\subsection{Brascamp-Lieb inequality}

In this subsection we recall  some inequalities proved in \cite{NZ} (see Lemmas 6.6 and 6.7 therein), which can be deduced from  the Brascamp-Lieb inequality (see   \cite{bl}) or just using
  H\"older's and Young's convolution inequalities. 

 \begin{lemma}  \label{lem2.1}
 Fix an integer $M\ge 2$. Let  $f$ be a non-negative function on the integers and set ${\bf k} = (k_1, \dots, k_M)$. Then,  for any vector ${\bf v} \in \R^M$  whose  components are $1$ or $-1$, we have
\begin{equation}  \label{equ6a}
  \sum_{  {\bf k} \in \mathbb{Z}^M}  f({\bf k}  \cdot  {\bf v} )   \prod _{j=1} ^M f(k_j)    \le C \left(\sum_{k\in \mathbb{Z}} f(k)^{1+ \frac 1M}\right)^M.
\end{equation}
\end{lemma}

 \begin{lemma} \label{lem2.2}
 Fix an integer $M\ge 3$ and assume  $\sum_{k \in \mathbb{Z}}
\rho(k)^2<\infty$.  We have
\begin{equation}  \label{equ22}
  \sum_{|k_j| \le n  \atop 1\le j \le M}   \rho(k_1)^2 |\rho( {\bf k}  \cdot  {\bf v} )|  \prod _{j=2} ^{M} | \rho(k_j) |   \le C \left(\sum_{|k| \leq n} |\rho(k)|\right)^{M-2},
\end{equation}
where ${\bf k} = (k_1, \dots, k_M)$ and ${\bf v} \in \R^M$ is a fixed vector whose  components are  $0
$, $1$ or $-1$ and it has at least two nonzero components.   
\end{lemma}

 \subsection{Gebelein's inequality}\label{ss:geb}

In the proof of Theorem \ref{thm1}, we will need the following Gaussian inequality.
 
 \begin{lemma} \label{Geb}
 Let  $W = \{ W(h), h\in \HH\} $ be an isonormal Gaussian
process over some real separable Hilbert space  $\HH$, and let  $\HH_1$ , $\HH_2$ be
two Hilbert subspaces of $\HH$. Define $W_1$ and $W_2$, respectively, to be the
restriction of $W$ to $\HH_1$ and $\HH_2$. Now consider two measurable mappings
$F_i:  \RR^{\HH_i} \rightarrow \RR$, $i=1,2$,  and assume that each  $F_i(W_i)$ is centered and
 $F_1(W_1) \in L^p(\Omega)$, $F_2(W_2) \in L^q(\Omega)$,  with $\frac 1p + \frac 1q=1$.
Then,
 \[
 |\E[ F_1(W_1) F_2(W_2)]| \le  \theta  \| F_1(W_1)\|_p \| F_2(W_2)\|_q,
 \]
 where
\[
 \theta := \sup_{ g\in  \HH_1, h\in  \HH_2,  \|g\|=\|h\|=1}  | \langle h,g \rangle _{\HH}|.
\] 
 \end{lemma}
 
 Lemma \ref{Geb} follows from the forthcoming Proposition \ref{p:k}, and can be shown by adopting almost verbatim the strategy of proof of \cite[Theorem 3.4]{V} -- details are left to the reader.

 \begin{proposition}\label{p:k}
 Let  $W = \{ W(h) : h\in \HH\} $,  $\widehat{W} = \{ \widehat{W}(h) : h\in \HH\} $  two independent  isonormal Gaussian
processes over some real separable Hilbert space  $\HH$.
 Consider two measurable mappings
$F_i:  \RR^{\HH} \rightarrow \RR$, $i=1,2$,  and assume that each  $F_i(W)$ is centered and
 $F_1(W) \in L^p(\Omega)$, $F_2(W) \in L^q(\Omega)$,  with $\frac 1p + \frac 1q=1$.
Then, for any $\theta \in [-1,1]$,
 \[
 |\E[ F_1(W) F_2(\theta W + \sqrt{1-\theta^2}     \widehat{W})]| \le C  |\theta | \| F_1(W)\|_p \| F_2(W)\|_q,
 \]
 for some constant $C$ depending uniquely on $p$.
 \end{proposition}
 
 \begin{proof}[Proof of Proposition \ref{p:k}]
 Without loss of generality, we can assume that $\theta \in (0,1)$. Using Mehler's formula (see e.g. \cite[formula (1.67), p. 55]{nualartbook}) together with the properties of conditional expectations, we infer that
 \begin{align*}
 \E[ F_1(W) F_2(\theta W + \sqrt{1-\theta^2}     \widehat{W})] = \E[ F_1(W) P_{\log\frac1\theta} F_2(W)],
 \end{align*}
 where $\{ P_t : t\geq 0\}$ is the Ornstein-Uhlenbeck semigroup introduced above. The conclusion now follows from a standard application of the Cauchy-Schwarz inequality, as well as from the following estimate: for every $q>1$ and every $u>0$,
 $$
 \| P_u F_2(W)\|_q\leq C e^{-u} \| F_2(W)\|_q,
 $$
 for some constant $C$ uniquely depending on $q$, which follows from a direct application of \cite[Lemma 1.4.1]{nualartbook}, as well as from the fact that $F_2$ is centered by assumption.
%
%
 \end{proof}

 \section{Proof of Theorem \ref{thm1}} 

We are now ready for the proof of Theorem \ref{thm1}. In what follows, we use the letter $C>0$ to indicate a constant that may depend on the $\DD^{1,4}(\R,\gamma)$ norm of $g$, but which is always independent of $n$. Its exact value is immaterial and may vary from one line to another.
The main difficulty of the proof is to show the forthcoming inequality (\ref{toshow}). 
 
\bigskip

\underline{\it Step 1: Preparing the proof}. We shall use the Malliavin-Stein approach. In order to be in a position to do so, 
consider a centered stationary Gaussian  family of random variables   $X=\{ X_n , n\ge 0\}$  with unit variance and covariance
    $\rho(k) = \EE(X_0 X_k)$ for $k \ge 0$.  We put $\rho(-k) = \rho(k)$ for $k<0$.
    Suppose that $\HH$ is a Hilbert space and let $\{e_i,\,i \ge 0\}$  be a family of $\HH$ such that
    $\langle e_i, e_j \rangle_{\HH}= \rho(i-j)$ for each $i,j \ge 0$. In this situation, if  $\{ W(\phi): \phi \in \HH\}$ is an isonormal Gaussian process, then
      the sequence  $X=\{ X_n , n\ge 0\}$  has the same law as $\{ W(e_n), n \ge 0\}$ and we can assume, without any loss of generality, that
      $X_n = W(e_n)$. 

      Consider the sequence   $	  F_n := \frac{1}{\sqrt{n}} \sum_{j=1}^n g(X_j) $ introduced in (\ref{fn}), where $g\in L^2(\R, \gamma)$ has  Hermite rank $d\ge 2$ and let 
	$\sigma_n^2= \EE (F_n^2)$.    Under condition  (\ref{h1}), it is well known that $\sigma_n^2 \to \sigma^2$ as $n\to \infty$,  where $\sigma^2$ has been defined in (\ref{bm.sig}). 
	Set $Y_n =\frac {F_n} {\sigma_n}$.  Notice that $\sigma>0$ implies that  $\sigma_n$ is bounded below for $n$ large enough. 
	Taking into account  (\ref{g.intrep}), we have  the representation $Y_n = \delta (\frac 1{ \sigma_n}u_n)$, where 
\begin{equation} \label{un}
u_n  =   \frac{1}{\sqrt{n}} \sum_{j=1}^n g_1(X_j) e_j,
\end{equation}
and $g_1$ is the shifted function introduced in  (\ref{t1a}). As a consequence of Proposition \ref{prop2}, we have the estimate
\begin{align}
	d_{TV} (Y_n ,N) & \leq  2\sqrt{ {\rm Var} ( \langle DY_n,  \frac 1{\sigma_n}u_n \rangle_{\HH})}  \le C \sqrt{ {\rm Var} ( \langle DF_n,  u_n \rangle_{\HH})},  \label{yn2.est}
\end{align}
for an absolute constant $C$. We now observe that there exists a sequence $\{g^{[m]} : m\geq 1\}\subset  \mathbb{D}^{3,4}(\RR, \gamma)$ such that $g^{[m]} \to g$ in the  $\mathbb{D}^{1,4}(\RR, \gamma)$ topology. For such a sequence of functions it is easily checked that, as $m\to\infty$
\begin{equation} \label{ecu1bis}
\|g^{[m]}\|_{1,4} \to \|g\|_{1,4}, \quad  \|g_1^{[m]}\|_{2,4} \to \|g_1\|_{2,4}.
\end{equation}
Moreover, denoting by $K(m,n)$ the quantity obtained from ${\rm Var} ( \langle DF_n,  u_n \rangle_{\HH})$ by replacing $g$ with $g^{[m]}$ one has that, as $m\to \infty$ and for each fixed $n$,
$$
K(m,n) \to {\rm Var} ( \langle DF_n,  u_n \rangle_{\HH}).
$$
This  follows from the fact that  for each $j\ge 1$, the sequences $g_1^{[m]}(X_j)$ and $(g^{[m]})'(X_j)$ converge in  $L^4(\Omega)$, as $m$ tends to infinity, to $g_1(X_j)$ and $g'(X_j)$, respectively, due to
the convergences (\ref{ecu1bis}).

The rest of the proof will then consist in showing that, for every function $g\in \mathbb{D}^{3,4}(\RR, \gamma)$,
\begin{equation}\label{toshow}
{\rm Var} ( \langle DF_n,  u_n \rangle_{\HH})\leq 
C n^{-1} \,\sum_{|k| \leq n} |\rho(k)| + C n^{-1} \left(\sum_{|k| \leq n} |\rho(k)|^{\frac{4}{3}}\right)^3,
\end{equation}
for constants $C$ that only depend on the $\mathbb{D}^{1,4}(\RR, \gamma)$ norm of $g$ and on the $\mathbb{D}^{2,4}(\RR, \gamma)$ norm of $g_1$   (recall that, by (\ref{ecu2}), $\|g_1\|_{2,4}\leq C\|g\|_{1,4}$).
\medskip

\underline{\it Step 2: Bounding ${\rm Var} ( \langle DF_n,  u_n \rangle_{\HH})$}. 
We have
 \begin{align*}
\Phi_n:=    \langle DF_n,u_n \rangle_{\HH} =\frac{1}{n} \sum_{i,j=1}^n g'(X_i)g_1(X_j)\rho(i-j).
    \end{align*}
    We can write
    \[
    \langle DF_n,u_n \rangle_{\HH} - \E(\langle DF_n,u_n \rangle_{\HH} )= \delta(- DL^{-1} \Phi_n)=:\delta(v_n).
    \]
    We will make use of the following estimate
    \begin{equation}\label{es314}
    {\rm Var} ( \langle DF_n,  u_n \rangle_{\HH})= \E( | \delta(v_n)|^2) \le   \| \E (v_n)\|^2_{\HH}  +   2 \E( \|  Dv_n\|^2_{\HH^{\otimes 2}}),
    \end{equation}
which can be justified as follows. First, by the isometry formula one has 
$ \E( | \delta(v_n)|^2) =   \E (\| v_n\|^2_{\HH})  +    \E( \|  Dv_n\|^2_{\HH^{\otimes 2}})$. Then, one can write
$ \E (\| v_n\|^2_{\HH}) =  \E (\| v_n - \E (v_n)\|^2_{\HH}) +  \| \E (v_n)\|^2_{\HH}$ and then apply Poincar\'e formula to the first term in the right-hand side to obtain (\ref{es314}).
  We will now proceed with the estimation of each member of the right-hand side of (\ref{es314}).

\medskip

\underline{\it Step 3: Estimating $\| \E (v_n)\|^2_{\HH}  $}. 
We first note that $\E(-DL^{-1}Z)=\E(DZ)$ for any $Z\in L^2(\Omega)$, as is immediately seen by expanding $Z$ into chaos.
We then have
\begin{align*}
\| \E (v_n)\|^2_{\HH}  & =\|  \E[D( \langle DF_n,u_n \rangle_{\HH})]\|^2_{\HH} \\
&= \frac 1 {n^2}   \sum_{i_1, \dots, i_4=1}^n  \langle \E[ D( g'(W(e_{i_1})) g_1 (W(e_{i_2}))) ] , \E[ D( g'(W(e_{i_3})) g_1 (W(e_{i_4 }))) ] \rangle_{\HH}\\
& \qquad \times
  \rho(i_1-i_2) \rho(i_3-i_4)  \\
  &= \frac 1 {n^2}   \sum_{i_1, \dots, i_4=1}^n  \Bigg(  \E [g''(W(e_{i_1})) g_1 (W(e_{i_2}))]  \E  [g''(W(e_{i_3})) g_1 (W(e_{i_4 })) ] \\
& \qquad \times
 \rho(i_1-i_3) \rho(i_1-i_2) \rho(i_3-i_4)  \\
 & \qquad +\E [g'(W(e_{i_1})) (g_1)' (W(e_{i_2}))]  \E  [g''(W(e_{i_3})) g_1 (W(e_{i_4 })) ] 
 \rho(i_2-i_3) \rho(i_1-i_2) \rho(i_3-i_4)  \\
  & \qquad +\E [g''(W(e_{i_1})) g_1 (W(e_{i_2}))]  \E  [g'(W(e_{i_3})) (g_1)' (W(e_{i_4 })) ] 
 \rho(i_1-i_4) \rho(i_1-i_2) \rho(i_3-i_4)  \\
   & \qquad +\E [g'(W(e_{i_1})) (g_1)' (W(e_{i_2}))]  \E  [g'(W(e_{i_3})) (g_1)' (W(e_{i_4 })) ] 
 \rho(i_2-i_4) \rho(i_1-i_2) \rho(i_3-i_4) \Bigg).
\end{align*}
Notice that we have three covariance factors. We need two additional factors that will be produced by the representation as a divergence of
$g'(W(e_i))$ and $g_1(W(e_i))$. That is, we can write
\begin{align*}
\E [g''(W(e_{i_1})) g_1 (W(e_{i_2}))]& =\E [g''(W(e_{i_1})) \delta (g_2 (W(e_{i_2})) e_2)] \\
&=\E [g^{(3)}(W(e_{i_1})) g_2 (W(e_{i_2}))] \rho(i_1-i_2).
\end{align*}
and
\begin{align*}
\E [g'(W(e_{i_1})) (g_1)' (W(e_{i_2}))]& =\E [\delta(g'_1(W(e_{i_1})) e_1) (g_1)' (W(e_{i_2})) ] \\
&=\E [ g'_1(W(e_{i_1})) (g_1)'' (W(e_{i_2}))] \rho(i_1-i_2).
\end{align*}
We claim that the expectations  $\E [g^{(3)}(W(e_{i_1})) g_2 (W(e_{i_2}))]$ and $\E [ g'_1(W(e_{i_1})) (g_1)'' (W(e_{i_2}))] $ are bounded.
 Indeed, using the expansion of $g$ in Hermite polynomials, we have
 \begin{align*}
| \E [g^{(3)}(W(e_{i_1})) g_2 (W(e_{i_2}))] |&=  \left|\E \left( \sum_{q=3} ^\infty q(q-1)(q-2)  c_q H_{q-3} (W(e_{i_1}))
  \sum_{q=2} ^\infty   c_q H_{q-2} (W(e_{i_2})) \right) \right| \\
  & =  \left|\sum_{q=3} ^\infty  c_q c_{q-1} q(q-1)(q-2) (q-3)!   \rho^{q-3}(i_1-i_2) \right|\\
  &\le  \sum_{q=3} ^\infty  |c_q c_{q-1}| q! 
  \le 
 \left[ \sum_{q=3} ^\infty  c^2_q q! \sum_{q=1} ^\infty c^2_{q}  (q+1)q! \right] ^{\frac 12},
 \end{align*}
 which is finite because
 the last quantity is precisely $   \| g\|_{L^2(\R,\gamma)}  \|g\|_{\DD^{1,2} (\R,\gamma)} $.

 The term $\E [ g'_1(W(e_{i_1})) (g_1)'' (W(e_{i_2}))] $ can be handled in the same way.
 As a consequence,
 \begin{align*}
\| \E (v_n)\|^2_{\HH}  & \le 
 \frac C {n^2}   \sum_{i_1, \dots, i_4=1}^n   (|\rho|(i_1-i_3) +|\rho|(i_2-i_3) +|\rho|(i_1-i_4)+| \rho|(i_2-i_4) ) \rho^2(i_1-i_2) \rho^2(i_3-i_4)  \\
& \le \frac C n  \sum_{|k| \le n} |\rho(k)|.
 \end{align*}

\medskip

\underline{\it Step 4: Estimating  $ \E( \|  Dv_n\|^2_{\HH^{\otimes 2}})$}.    We have
\begin{align*}
\alpha_n:&= \E( \|  Dv_n\|^2_{\HH^{\otimes 2}})= \E( \|  D^2 L^{-1} \langle DF_n, u_n \rangle_{\HH} \|^2_{\HH^{\otimes 2}})  
=\E( \|  K L^{-1}  D^2  \langle DF_n, u_n \rangle_{\HH} \|^2_{\HH^{\otimes 2}}),
\end{align*}
where $K$ is the operator defined by $KG=\sum_{m=0}^\infty \frac{m}{m+2}I_m(g_m)$ for $G=\sum_{m=0}^\infty I_m(g_m)\in L^2(\Omega)$.
Since $K$ is bounded in $L^p(\Omega)$ for all $p>1$ (see \cite[Theorem 1.4.2]{nualartbook}), we obtain
\[
\alpha_n  \le C      \E( \|   L^{-1}  D^2  \langle DF_n, u_n \rangle_{\HH} \|^2_{\HH^{\otimes 2}})
\]
We have
\begin{align*}
L^{-1}D^{2} ( \langle DF_n,u_n \rangle_{\HH} ) &= \frac 1n  \sum_{i,j=1}^n   L^{-1}D^2[  g'(W(e_i)) g_1(W(e_j))] \rho(i-j)\\
& = \frac 1n \sum_{k=0}^2  {2\choose k}\sum_{i,j=1}^n   L^{-1} [ g^{(k+1)} (W(e_i)) (g_1)^{(2-k)} (W(e_j))] e_i^{\otimes k} \otimes e_j^{\otimes (2-k)}\rho(i-j).
\end{align*}
Therefore,
\begin{align*}
& \E \left( \| L^{-1}D^{2} ( \langle DF_n,u_n \rangle_{\HH} )  \| ^2_{\HH^{\otimes2}} \right)\\
 &\le  \frac C  {n^2}  \sum_{k=0}^2   \E \left( \left \| \sum_{i,j=1}^n   L^{-1} [ g^{(k+1)} (W(e_i)) (g_1)^{(2-k)} (W(e_j))] e_i^{\otimes k} \otimes e_j^{\otimes (2-k)}\rho(i-j)  \right \| ^2_ {\HH^{\otimes 2} }\right)\\
 &= \frac C {n^2}  \sum_{k=0}^2    \sum_{i_1, i_2,i_3,i_4=1}^n   \E \left[  L^{-1} [ g^{(k+1)} (W(e_{i_1})) (g_1)^{(2-k)} (W(e_{i_2} )) ]  L^{-2} [ g^{(k+1)} (W(e_{i_3})) (g_1)^{(2-k)} (W(e_{i_4} )) ]
  \right]  \\
  &\qquad \times  \rho^k(i_1-i_3) \rho^{2-k}(i_2-i_4)\rho(i_1-i_2)  \rho(i_3-i_4) \\
  &=:   \sum_{k=0}^2  \frac C {n^2} \sum_{i_1, i_2,i_3,i_4=1}^n \beta_{k,n}  \rho^k(i_1-i_3) \rho^{2-k}(i_2-i_4)\rho(i_1-i_2)  \rho(i_3-i_4).
\end{align*}
We split the analysis on the different values of $k$.

\medskip
\noindent
{\it Case $k=1$}. \quad  We have
\begin{align*}
|\beta_{1,n}|&= |\E \left[  L^{-1} [ g'' (W(e_{i_1})) (g_1)' (W(e_{i_2} )) ]  L^{-1} [ g'' (W(e_{i_3})) (g_1)' (W(e_{i_4} )) ]
  \right] | \\
  &\le  \|L^{-1} [ g'' (W(e_{i_1})) (g_1)' (W(e_{i_2} )) ] \|_2  \|L^{-1} [ g'' (W(e_{i_3})) (g_1)' (W(e_{i_4} )) ]\|_2.
  \end{align*}
We can write
\[
g'' (W(e_{i_1}))=  \langle D[g' (W(e_{i_1}))], e_1 \rangle_{\HH}.
\]
As a consequence,
\begin{align*}
\|L^{-1} [ g'' (W(e_{i_1})) (g_1)' (W(e_{i_2} )) ] \|_2& =\| \langle   L^{-1} ( D[g' (W(e_{i_1}))] (g_1)' (W(e_{i_2} )) ) , e_1 \rangle_{\HH}  \|_2 \\
&\le  \|    L^{-1} ( D[g' (W(e_{i_1}))] (g_1)' (W(e_{i_2} )) ) \|_{ L^2(\Omega; \HH)}.
\end{align*}
This quantity is uniformly bounded  by a constant times $\| g\|_{\DD^{1,4} (\RR, \gamma)}
\| g_1\|_{\DD^{2,4} (\RR, \gamma)}$, due to Lemma \ref{lem1}  (1)  applied to
$F:=g' (W(e_{i_1})) \in L^4(\Omega)$ and $G:=(g_1)' (W(e_{i_2} )) \in \DD^{1,4}$
and taking into account that
$ \|F\|_4 \le \| g\|_{\DD^{1,4} (\RR, \gamma)} $ and 
\[
\| G\|_{1,4} = \| (g_1)'(W(e_{i_2})) \|_{1,4}   \le  \| g_1(W(e_{i_2})) \|_{2,4}  = \| g_1\|_{\DD^{2,4} (\RR, \gamma)}.
\]
    Therefore,
    \begin{align*}
 &  \frac 1{n^2}   \sum_{i_1, i_2,i_3,i_4=1}^n  |\beta_{1,n}|  | \rho(i_1-i_3) \rho(i_2-i_4)\rho(i_1-i_2)  \rho(i_3-i_4)|  \\
 & \quad \le  \frac 1n \sum_{|k_i| \le n, i=1,2,3}  |\rho(k_1) \rho(k_2)\rho(k_3)\rho(k_2+k_3-k_1)| \le \frac  C {n}  \left( \sum_{|k| \le n} |\rho(k)| ^{\frac 43} \right)^3,
    \end{align*}
    where the last inequality follows from  Lemma \ref{lem2.1}.
    
    \medskip
\noindent
{\it Case $k=0$}. \quad  We have
\begin{align*}
|\beta_{0,n}|&= |\E \left[  L^{-1} [ g' (W(e_{i_1})) (g_1)'' (W(e_{i_2} )) ]  L^{-1} [ g' (W(e_{i_3})) (g_1)'' (W(e_{i_4} )) ] |
  \right] | \\
  &= |\E \left[  g' (W(e_{i_1})) (g_1)'' (W(e_{i_2} ))   L^{-2} [ g' (W(e_{i_3})) (g_1)'' (W(e_{i_4} )) ] \right] |.
  \end{align*}
  We know that   $ g' (W(e_{i_1}))$ is centered and belongs to $L^4(\Omega)$. Moreover,
  \[
  (g_1)'' (W(e_{i_2} ))   L^{-2} [ g' (W(e_{i_3})) (g_1)'' (W(e_{i_4} )) ]
  \]
  belongs to $L^{\frac 43}(\Omega)$.  Indeed, using H\"older inequality, we can write
  \begin{align*}
 & \| (g_1)'' (W(e_{i_2} ))   L^{-2} [ g' (W(e_{i_3})) (g_1)'' (W(e_{i_4} ))] \|_{\frac 43}  \\
 & \quad \le
    \| (g_1)'' (W(e_{i_2} )) \|_4 \|  L^{-2} [ g' (W(e_{i_3})) (g_1)'' (W(e_{i_4} ))] \|_{2} \\
    &\quad \le   C   \| (g_1)'' (W(e_{i_2} )) \|_4  \| g' (W(e_{i_3}))\|_4 \| (g_1)'' (W(e_{i_4} ))\|_4 \\
    & \quad  \le C \| g_1 (W(e_{i_2} )) \|_{2,4}  \| g (W(e_{i_3}))\|_{1,4} \| g_1 (W(e_{i_4} ))\|_{2,4}\\
    & \quad = C \| g\|_{\DD^{1,4} (\RR,\gamma)}\| g_1\|^2_{\DD^{2,4} (\RR,\gamma)}.
   \end{align*} 
  Therefore, by   Gebelein's inequality (see Lemma \ref{Geb}), we deduce
  \begin{align*}
  |\beta_{0,n}|  &\le   (|\rho( (i_1-i_2)|+ |\rho( (i_1-i_3)| + |\rho( (i_1-i_4)|)  \\
  &   \qquad \times  \| g' (W(e_{i_1}))\|_4  \| (g_1)'' (W(e_{i_2} ))   L^{-2} [ g' (W(e_{i_3})) (g_1)'' (W(e_{i_4} ))]\|_{4/3} \\ 
  & \le C  (|\rho( (i_1-i_2)|+ |\rho( (i_1-i_3)| + |\rho( (i_1-i_4)|).
  \end{align*}
    Therefore,
    \begin{align*}
 &  \frac 1{n^2}   \sum_{i_1, i_2,i_3,i_4=1}^n  |\beta_{0,n}|  |  \rho(i_2-i_4)^2\rho(i_1-i_2)  \rho(i_3-i_4)|  \\
  & \le   \frac C{n^2}   \sum_{i_1, i_2,i_3,i_4=1}^n   (|\rho( (i_1-i_2)|+ |\rho( (i_1-i_3)| + |\rho( (i_1-i_4)|)    |  \rho(i_2-i_4)^2\rho(i_1-i_2)  \rho(i_3-i_4)| \\
 & \quad =  \frac C{n^2}\sum_{i_1, i_2,i_3,i_4=1}^n        \rho^2(i_2-i_4)\rho^2(i_1-i_2) | \rho(i_3-i_4)|  \\
  & \quad  \qquad +  \frac C{n^2}\sum_{i_1, i_2,i_3,i_4=1}^n        \rho^2(i_2-i_4) | \rho(i_1-i_3) \rho(i_1-i_2)   \rho(i_3-i_4)|  \\
    & \quad  \qquad +  \frac C{n^2}\sum_{i_1, i_2,i_3,i_4=1}^n        \rho^2(i_2-i_4) | \rho(i_1-i_4) \rho(i_1-i_2)   \rho(i_3-i_4)| \\
    & \quad =: A_n +B_n +C_n.
        \end{align*}
        For $A_n$, we have
        \[
        A_n \le   \frac Cn  \sum_{|k| \le n} | \rho(k)|.
        \]
        The terms $B_n$ and $C_n$ are similar. For $B_n$, we have
         \begin{align*}   
B_{n } &\le  \frac Cn  \sum_{|k_i| \le n, i=1,2,3}  | \rho(k_1)\rho(k_2) \rho(k_3)|\rho^2(k_1+ k_2+k_3) \le \frac  C {n}   \sum_{|k| \le n} | \rho(k)|,
    \end{align*}
    where we have applied Lemma  \ref{lem2.2} in the last inequality.
    
      \medskip
\noindent
{\it Case $k=2$}. \quad  We have
\begin{align*}
\beta_{2,n}&= \E \left[  L^{-1} [ g^{(3)} (W(e_{i_1})) g_1(W(e_{i_2} )) ]  L^{-1} [ g^{(3)} (W(e_{i_3})) g_1 (W(e_{i_4} )) ] 
  \right] | \\
  &=\E \left[  g^{(3)} (W(e_{i_1})) g_1 (W(e_{i_2} ))   L^{-2} [ g^{(3)} (W(e_{i_3})) g_1 (W(e_{i_4} )) ] \right]  \\
  & = \E \left[    L(g_1 (W(e_{i_2} ))) L^{-1} \{ g^{(3)} (W(e_{i_1}))     L^{-2} [ g^{(3)} (W(e_{i_3})) g_1 (W(e_{i_4} )) ]  \}\right] .
  \end{align*}
We know that     $ L(g_1 (W(e_{i_2})))$ is centered and 
 \begin{equation} \label{ecu6}
 \| L(g_1 (W(e_{i_2}))) \|_4 \le C \| g_1\| _{\DD^{2,4} (\R,\gamma)}.
 \end{equation}
  Moreover,
the random variable $   L^{-1} \{ g^{(3)} (W(e_{i_1}))     L^{-2} [ g^{(3)} (W(e_{i_3})) g_1 (W(e_{i_4} )) ]  \}$
  belongs to $L^{\frac 43}(\Omega)$. Indeed, its $L^{\frac 43}(\Omega)$-norm can be estimated as follows
  \begin{align}  \notag
& \|     L^{-1} \{ g^{(3)} (W(e_{i_1}))     L^{-2} [ g^{(3)} (W(e_{i_3})) g_1 (W(e_{i_4} )) ]  \} \| _{\frac 43} \\ \notag
&= \|     L^{-1} \{  \langle D^2(g' (W(e_{i_1}))) , e_{i_1}  \otimes e_{i_1}\rangle_{\HH}     L^{-2} [ g^{(3)} (W(e_{i_3})) g_1 (W(e_{i_4} )) ]  \} \| _{\frac 43} \\  \label{ecu3}
& \le  \|     L^{-1} \{  D^2(g' (W(e_{i_1})))     L^{-2} [ g^{(3)} (W(e_{i_3})) g_1 (W(e_{i_4} )) ]  \} \| _{L^\frac 43(\Omega ; \HH^{\otimes 2})}.
\end{align} 
  By Lemma \ref{lem1} (2) applied to
 $F= g'(W(e_{i_1})  $ and $G=   L^{-2} [ g^{(3)} (W(e_{i_3})) g_1 (W(e_{i_4} )) ]$, we have
  \begin{align} \notag
&\|     L^{-1} \{  D^2(g' (W(e_{i_1})))     L^{-2} [ g^{(3)} (W(e_{i_3})) g_1 (W(e_{i_4} )) ]  \} \| _{L^\frac 43(\Omega ; \HH^{\otimes 2})} \\  \label{ecu4}
&\le C \| g' (W(e_{i_1})) \| _4 \| L^{-2} [ g^{(3)} (W(e_{i_3})) g_1 (W(e_{i_4} )) ]\|_{2,2}.
\end{align}
Then Meyer inequalities (see (\ref{meyer})) imply that
\[
\| L^{-2} [ g^{(3)} (W(e_{i_3})) g_1 (W(e_{i_4} )) ]\|_{2,2} \le
  C   \| L^{-1} [ g^{(3)} (W(e_{i_3})) g_1 (W(e_{i_4} )) ]\|_{2}.
\]
We can write
\begin{align*}
\| L^{-1} [ g^{(3)} (W(e_{i_3})) g_1 (W(e_{i_4} )) ]\|_{2}& =
\| L^{-1} [  \langle D^2 (g' (W(e_{i_3})), e_{i_3} \otimes e_{i_3} \rangle_{\HH^{\otimes 2}} g_1 (W(e_{i_4} )) ]\|_{2} \\ 
&\le  \| L^{-1} [    D^2 (g' (W(e_{i_3})) g_1 (W(e_{i_4} )) ]\|_{L^2(\Omega; \HH^{\otimes 2} }.
\end{align*}
Then, a further application of Lemma \ref{lem1}  (2) to $F=g' (W(e_{i_3})$ and $G=g_1 (W(e_{i_4} ))$, yields
\begin{equation} \label{ecu5}
\| L^{-1} [    D^2 (g' (W(e_{i_3})) g_1 (W(e_{i_4} )) ]\|_{L^2(\Omega; \HH^{\otimes 2} }\le
 \| g' (W(e_{i_3})\|_4 \|  \| g_1 (W(e_{i_4} ))\|_{2,4}.
 \end{equation}
Thus, from (\ref{ecu3}), (\ref{ecu4}) and  (\ref{ecu5}) we deduce
\begin{equation} \label{ecu7}
 \|     L^{-1} \{ g^{(3)} (W(e_{i_1}))     L^{-2} [ g^{(3)} (W(e_{i_3})) g_1 (W(e_{i_4} )) ]  \} \| _{\frac 43}
 \le C    \| g\| ^2_{\DD^{1,4} (\R,\gamma)}  \| g_1\| _{\DD^{2,4} (\R,\gamma)}.
 \end{equation}
  Therefore, by   Gebelein's inequality (see Lemma \ref{Geb}), and the bounds  (\ref{ecu6}) and (\ref{ecu7}), we obtain
  \begin{align*}
  |\beta_{2,n}|  &\le   (|\rho( i_1-i_2)|+ |\rho( i_1-i_3)| + |\rho( i_1-i_4)|)  \\
  &   \qquad \times  \| L(g_1 (W(e_{i_1})))\|_4  \|L^{-1} [ g^{(3)} (W(e_{i_2} ))   L^{-2} [ g^{(3)}(W(e_{i_3})) g_1 (W(e_{i_4} ))]]\|_{4/3}\\
  & \le C  
  (|\rho( i_1-i_2)|+ |\rho( i_1-i_3)| + |\rho( i_1-i_4)|).
  \end{align*}
As a  consequence,
    \begin{align*}
 &  \frac 1{n^2}   \sum_{i_1, i_2,i_3,i_4=1}^n  |\beta_{2,n}|  |  \rho(i_2-i_4)^2\rho(i_1-i_2)  \rho(i_3-i_4)|  \\
  & \le   \frac C{n^2}   \sum_{i_1, i_2,i_3,i_4=1}^n   (|\rho( i_1-i_2)|+ |\rho( i_1-i_3)| + |\rho( i_1-i_4)|)    |  \rho(i_2-i_4)^2\rho(i_1-i_2)  \rho(i_3-i_4)| \\
 & \quad =  \frac C{n^2}\sum_{i_1, i_2,i_3,i_4=1}^n        \rho^2(i_2-i_4)\rho^2(i_1-i_2) | \rho(i_3-i_4)|  \\
  & \quad  \qquad +  \frac C{n^2}\sum_{i_1, i_2,i_3,i_4=1}^n        \rho^2(i_2-i_4) | \rho(i_1-i_3) \rho(i_1-i_2)   \rho(i_3-i_4)|  \\
    & \quad  \qquad +  \frac C{n^2}\sum_{i_1, i_2,i_3,i_4=1}^n        \rho^2(i_2-i_4) | \rho(i_1-i_4) \rho(i_1-i_2)   \rho(i_3-i_4)| \\
    & \quad =: A_n +B_n +C_n.
        \end{align*}
        For $A_n$, we have
        \[
        A_n \le   \frac Cn  \sum_{|k| \le n} | \rho(k)|.
        \]
        The terms $B_n$ and $C_n$ are similar. For $B_n$, we have
         \begin{align*}   
B_{n } &\le  \frac Cn  \sum_{|k_i| \le n, i=1,2,3}  | \rho(k_1)\rho(k_2) \rho(k_3)|\rho^2(k_1+ k_2+k_3)\\
 & \quad \le \frac  C {n}   \sum_{|k_i| \le n} | \rho(k)|,
    \end{align*}
    where we have applied Lemma  \ref{lem2.2} in the last inequality.

\medskip

\underline{\it Step 5: end of the proof}. From Step 1, it suffices to show that $${\rm Var} ( \langle DF_n,  u_n \rangle_{\HH})\leq 
C n^{-1} \,\sum_{|k| \leq n} |\rho(k)| + C n^{-1} \left(\sum_{|k| \leq n} |\rho(k)|^{\frac{4}{3}}\right)^3.$$
By Step 2, we have  ${\rm Var} ( \langle DF_n,  u_n \rangle_{\HH})\leq \| \E (v_n)\|^2_{\HH} + 2\E(\|Dv_n\|^2_{\HH^{\otimes 2}})$.
In Step 3, it is shown that $\| \E (v_n)\|^2_{\HH}  \le \frac C n  \sum_{|k| \le n} |\rho(k)|$.
Finally, it is shown in Step 4 that $\| \E (v_n)\|^2_{\HH}  \le C n^{-1} \,\sum_{|k| \leq n} |\rho(k)| + C n^{-1} \left(\sum_{|k| \leq n} |\rho(k)|^{\frac{4}{3}}\right)^3$.
The proof of Theorem \ref{thm1} is thus complete.

\qed

       \begin{remark}
    We can show that both bounds in (\ref{ecu1}) are not comparable. In the particular case $|\rho(k)| \sim |k|^{-\alpha}$ as $|k| \rightarrow \infty$, with $\alpha >\frac 12$, we obtain:
    $$
      d_{\rm TV}(Y_n , Z)  \le
      \left\{
         \begin{array}{ll}  
             Cn^{1-2\alpha } & {\rm if} \quad \frac 12 <\alpha < \frac 23, \\
            C n^{-\frac \alpha 2}  & {\rm if} \quad  \frac 23 \le \alpha <1, \\
            C n^{-\frac 12} (\log n)^{\frac 12} & {\rm if} \quad  \alpha = 1,\\
                      C n^{-\frac 12}  & {\rm if} \quad  \alpha > 1.
         \end{array}
      \right.
    $$
    \end{remark}

 \section*{Appendix}
 
 The following elementary result is used in the Introduction.
 
 \begin{lemma}\label{l:simple}
 Let $\{\rho(k) : k\in \mathbb{Z}\}\in \ell^2 $, and let $0<\alpha <2$ and $\beta, \gamma>0$ be such that
 \begin{equation}\label{e:el}
 \frac{2-\alpha}{2} \leq \frac\gamma\beta.
 \end{equation}
 Then,
 $$
 \lim_{n\to\infty} \frac{1}{n^\gamma} \left( \sum_{|k|< n} |\rho(k)|^\alpha\right)^\beta = 0.
 $$
 \end{lemma}
 \begin{proof}
 Write $T_n :=\frac{1}{n^\gamma} \left( \sum_{|k|< n} |\rho(k)|^\alpha\right)^\beta$, and let $m<n$. A straightforward application of H\"older inequality yields that, for some finite constant $C>0$ independent of $m,n$,
 \begin{eqnarray}
 \notag T_n &\leq& C\, \left\{ \frac{m^{(2-\alpha)/2}}{n^{\gamma/\beta}} + \frac{(n-m)^{(2-\alpha)/2}}{n^{\gamma/\beta}}\left( \sum_{|k|\geq m} |\rho(k)|^2\right)^{\alpha/2}\right\}^\beta\\ &\leq& C\, \left\{ \frac{m^{(2-\alpha)/2}}{n^{\gamma/\beta}} +\left( \sum_{|k|\geq m} |\rho(k)|^2\right)^{\alpha/2}\right\}^\beta,\label{e:qq}
\end{eqnarray}
where in the second inequality we have used \eqref{e:el}. Now fix $\epsilon>0$ and observe that, since $\rho \in\ell^2$, there exists an integer $m_0$ such that $$\left( \sum_{|k|\geq m_0} |\rho(k)|^2\right)^{\alpha/2} \leq \epsilon.$$ Setting $m=m_0$ in \eqref{e:qq} and letting $n\to\infty$, we eventually conclude that $\limsup_n T_n \leq C\epsilon^\beta$ for every $\epsilon>0$, and the conclusion follows.  
\end{proof}

\end{document}